\documentclass[12pt]{amsart}
\usepackage{amsmath,amsthm,amsfonts,amssymb,eucal, hyperref}

\newcommand{\<}{\langle}
\renewcommand{\>}{\rangle}

\renewcommand{\l}{\lambda}

\newcommand{\T}{\mathbb T}

\newcommand{\oq}{\ {\raise 7pt\hbox{${\scriptstyle\circ}$}}
\kern -7pt{
\hbox{$Q$}}}

\newcommand{\R}{ \mathbb R}
\newcommand{\Q}{ \mathbb Q}
\newcommand{\C}{ \mathbb C}

\renewcommand{\T}{\mathbb T}

\newcommand {\bze}{\mathbf 0}

\newcommand {\bn}{\mathbf n}

\newcommand{\rbar}{\overline{\mathbb R}}

\newcommand{\1}
{{\,\vrule depth3pt height9pt}{\vrule depth3pt height9pt}
{\vrule depth3pt height9pt}{\vrule depth3pt height9pt}\,}

\DeclareMathOperator {\dist} {{dist}}

\DeclareMathOperator{\supp}{{supp}}

\DeclareMathOperator{\ran}{{ran}}

\newtheorem{thm}{Theorem}[section]
\newtheorem{cor}[thm]{Corollary}
\newtheorem{cla}[thm]{Claim}
\newtheorem{lem}[thm]{Lemma}
\newtheorem{prop}[thm]{Proposition}

\theoremstyle{definition}
\newtheorem{defn}[thm]{Definition}

\newtheorem{rem}[thm]{Remark}

\numberwithin{equation}{section}

\newcommand{\bee}{\begin{equation}}
\newcommand{\ene}{\end{equation}}
\newcommand{\bees}{\begin{equation*}}
\newcommand{\enes}{\end{equation*}}
\newcommand{\bes}{\begin{split}}
\newcommand{\ens}{\end{split}}

\newcommand{\bet}{\begin{thm}}
\newcommand{\ent}{\end{thm}}
\newcommand{\bel}{\begin{lem}}
\newcommand{\enl}{\end{lem}}
\newcommand{\bec}{\begin{cor}}
\newcommand{\enc}{\end{cor}}
\newcommand{\becl}{\begin{cla}}
\newcommand{\encl}{\end{cla}}
\newcommand{\bep}{\begin{proof}}
\newcommand{\enp}{\end{proof}}
\newcommand{\ber}{\begin{rem}}
\newcommand{\enr}{\end{rem}}
\newcommand{\ep}{\varepsilon}

\newcommand{\Z}{\mathbb Z}

\renewcommand{\l}{\left}
\renewcommand{\r}{\right}

\DeclareMathOperator{\dom}{{Dom}}

\makeatletter
\def\square{\RIfM@\bgroup\else$\bgroup\aftergroup$\fi
  \vcenter{\hrule\hbox{\vrule\@height.6em\kern.6em\vrule}\hrule}\egroup}
\makeatother

 \usepackage{hyperref}

\newcommand{\one}{\mathbf{1}}

\usepackage[margin=1.1in]{geometry}

\begin{document}

\title[Gaps in the spectra]
{On gaps in the spectra of quasiperiodic Schr\"odinger operators with discontinuous monotone potentials}
\author[I. Kachkovskiy]
{Ilya Kachkovskiy}
\address{Department of Mathematics\\ Michigan State University\\
Wells Hall, 619 Red Cedar Rd\\ East Lansing, MI\\ 48824\\ USA}
\email{ikachkov@msu.edu}
\author[L. Parnovski]
{Leonid Parnovski}
\address{Department of Mathematics\\ University College London\\
Gower Street\\ London\\ WC1E 6BT\\ UK}
\email{leonid@math.ucl.ac.uk}
\author[R. Shterenberg]
{Roman Shterenberg}
\address{Department of Mathematics\\ University of Alabama, Birminghan\\
Campbell Hall\\1300 University Blvd\\ Birmingham, AL\\ 35294\\USA }
\email{shterenb@math.uab.edu}



\date{\today}



\begin{abstract}
We show that, for one-dimensional discrete Schr\"odinger operators, stability of Anderson localization under a class of rank one perturbations implies absence of intervals in spectra. The argument is based on well-known results of Gordon and del Rio--Makarov--Simon, combined with a way to consider perturbations whose ranges are not necessarily cyclic. The main application of the results is showing that a class of quasiperiodic operators with sawtooth-like potentials, for which such a version of stable localization is known, has Cantor spectra. We also obtain several results on gap filling under rank one perturbations for some general (not necessarily monotone) classes of quasiperiodic operators with discontinuous potentials.
\end{abstract}
\maketitle
\section{Introduction and statements of the results}
\subsection{Motivation and some of the main results}
Let $f\colon \R\to [-\infty,+\infty)$ be a $1$-periodic function, and $H(x)$ be a family of discrete quasiperiodic Schr\"odinger operators on $\ell^2(\Z)$:
\bee
\label{eq_h_def_qp}
(H(x)\psi)(n)=\psi(n+1)+\psi(n-1)+f(x+n\alpha)\psi(n),\quad x\in [0,1).
\ene
Let $\gamma>0$. We will say that $f$ is $\gamma$-{\it monotone} if
\bee
\label{eq_gamma_monotone}
f(y)-f(x)\ge \gamma(y-x),\quad 0\le x<y<1.
\ene
Two typical examples of $\gamma$-monotone potentials are $f(x)=\{x\}=x-\lfloor x \rfloor$ (the sawtooth potential) and $f(x)=\lambda \cot(\pi x)$, $\lambda>0$ (the Maryland model). We will say that $f$ is {\it sawtooth-type} if $f$ is $\gamma$-monotone and $f|_{[0,1)}$ extends to a (bounded) continuous map from $[0,1]$ to $\R$. Similarly, $f$ is {\it Maryland-type} if it is $\gamma$-monotone and $f|_{[0,1)}$ extends to a homeomorphism between $[0,1]$ and $[-\infty,+\infty]$.

In a series of works \cite{JK1,K,Xiaowen,JK2}, it is shown that a large class of one-dimensional quasiperiodic operators with $\gamma$-monotone potentials satisfies Anderson localization. As an example, suppose that $\alpha$ is Diophantine:
$$
\dist(n\alpha,\Z)\ge C|n|^{-\tau}, \quad \forall n\in \Z\setminus\{0\}.
$$
Then the results of \cite{JK1,K} show that Anderson localization holds for sawtooth-type and Maryland-type $f$ for almost all $x\in [0,1)$ (see Proposition \ref{prop_monotone_localization} and Remark \ref{rem_localizations} below for more complete statements).

Denote by $\Sigma_f$ the almost sure spectrum of the operator \eqref{eq_h_def_qp}. Shortly after the results of \cite{JK1} were announced, a natural question was raised: suppose that $f$ is sawtooth-type. Can $\Sigma_f$ contain an interval, or should it be a Cantor set? In the referee report on \cite{JK1}, it was suggested that, for sufficiently large $\gamma$ and perhaps under some smoothness assumptions on $f$, one should be able to show that $\Sigma_f$ is an interval.

One of our main results is Theorem \ref{th_main_cantor}, which provides some answers to this question in the regime of localization for $\gamma$-monotone potentials. It is convenient to postpone the complete formulation of this result until Section 2. However, we state the following as a separate theorem (in fact, contained in Corollary \ref{cor_interval} of the main Theorem \ref{th_main_cantor}).
\begin{thm}
\label{th_main_small}
Let $f$ be sawtooth-type function and suppose that $\alpha$ is Diophantine. Then $\sigma(H(x))$ is a Cantor set $($in particular, it contains no intervals$)$.
\end{thm}

We note that the question of Cantor spectrum for quasiperiodic operators has been an active research area. Some recent results in this direction include \cite{ten_martini,ten_martini_dry,beckus,DL,GJY}. See also the review \cite{fillman1} and references therein.

It is interesting to compare these results with those for Maryland-type potentials, which sometimes do not rely on localization and/or $\gamma$-monotonicity. We will call $f$ {\it weak Maryland-type} if $f|_{[0,1)}$ extends to a continuous map from $[0,1]$ to $[-\infty,+\infty]$ and $f(0)=-\infty$, $f(1-0)=+\infty$. We will also call $f$ {\it weak sawtooth-type} if $f|_{[0,1)}$ extends to a (bounded) continuous map from $[0,1]$ to $\R$, with
$$
-\infty<f(0)<f(1-0)<+\infty.
$$
It was shown in \cite{K} that, for weak Maryland-type $f$ and $\alpha\in \R\setminus\Q$, one has $\sigma(H(x))=\R$. One can argue by identifying $\R\cup\{\infty\}=\rbar$ with a circle, somewhat informally, that the range of $f$, in the case of weak sawtooth-type potentials, has a ``gap'' $\R\setminus f([0,1))$, which is not present in the Maryland-type case. In other words, one ``gap'' in the values of $f$ produces multiple gaps in the spectra. In the second part of our paper, we partially confirm this intuition by demonstrating the following {\it gap filling} phenomenon. Let $f$ be a weak sawtooth-type function. For $t\in \R\setminus [f(0),f(1-0)]$, consider the following family:
\bee
\label{eq_ft_def}
f_t(x):=\begin{cases}
f(x),&x\in (0,1);\\
t,&x=0.
\end{cases}
\ene
Let $H_t(x)$ be $H(x)$ with $f$ replaced by $f_t$:
\bee
\label{eq_ht_def_qp}
(H_t(x)\psi)(n)=\psi(n+1)+\psi(n-1)+f_t(x+n\alpha)\psi(n),\quad x\in [0,1),\quad t\in \R\setminus f((0,1)).
\ene
\begin{thm}
\label{th_gap_filling_small}
Let $f$ be a weak sawtooth-type function and $\alpha\in\R\setminus \Q$. Then
\bee
\label{eq_gap_filling_small}
\overline{\bigcup\limits_{\R\setminus [f(0),f(1-0)]}\sigma(H_t(0))}=\R.
\ene
\end{thm}
The closure in the left hand side of \eqref{eq_gap_filling_small} would not be necessary if one included $t=\infty$ into consideration by allowing infinite potentials. Again, we postpone the more general setting  until Section 4 (see, in particular, Theorems \ref{th_gap_filling} and \ref{th_gap_filling_2}), where this inclusion is performed systematically. In some situations, one can even prove that every point of $\R$ is an actual eigenvalue of some operator from our  family, see Corollary \ref{cor_all_eigenvalues}. Note that, if one considers the union over all $t\in \R$ in \eqref{eq_gap_filling_small}, the result immediately follows from spectral averaging (for example, \cite[Theorem 11.8]{Simon}), which is why we assume $f(0)<f(1-0)$.

The general version of Theorem \ref{th_gap_filling_small} also holds in higher dimensions for operators on $\Z^d$. A similar question about Theorem \ref{th_main_small} appears to be more difficult. While the general argument behind Theorem \ref{th_main_small} is abstract, it relies on two properties of the operator: Anderson localization that is stable under a family of rank one perturbations (that is, replacing $f$ by $f_t$) and the fact that a standard basis vector is cyclic on a sufficiently large subspace. In Theorem \ref{th_main_small}, both of these properties rely on the one-dimensional structure of the operators, the first one through \cite{JK2}, and the second one through Lemma \ref{lemma_main} below. The first question has been recently addressed in \cite{Shi,KPS3}.

\subsection{Stability of localization and absence of intervals}We will now state the main abstract result that will be used to prove Theorem \ref{th_main_small}. It relies on the following classical result that, as far as the authors are aware, has been discovered independently in \cite{G1,G2,drms}, see also \cite{djls}.
\begin{prop}
\label{prop_gordon}
Let $A$ be a self-adjoint operator on a Hilbert space with simple spectrum and a cyclic vector $\varphi$. Let $P$ be the orthogonal projection onto the span of $\{\varphi\}$. Then, for every $t$ in some dense $G_{\delta}$-subset of $\R$, all eigenvalues of the operator $A+tP$ are isolated.
\end{prop}
For Anderson-type Hamiltonians whose point spectra are unions of intervals, this result implies that Anderson localization is unstable under rank one perturbations: indeed, for the set of values of $t$ under consideration, $H$ will not have any eigenvalues on any such interval, and therefore the spectral measure must contain a continuous component. 

Our goal is a somewhat unusual application of this result in a reverse direction: if an operator has a ``very stable'' kind of localization, then its spectrum cannot contain an interval. More precisely, let $H$ be a discrete Schr\"odinger operator on $\ell^2(\Z)$:
\bee
\label{eq_h_def}
H=\Delta+V,\quad (H\psi)(n)=\psi(n+1)+\psi(n-1)+V(n)\psi(n)
\ene
with real-valued potential $V$. Let also $P$ be the orthohonal projection operator on the standard basis vector $e_0$: that is,
$$
(P\psi)(n)=\psi(0)e_0.
$$
The following is our main abstract result.
\begin{thm}
\label{th_main}
Define $H$ and $P$ as above. Let $J\subset \R$ be an open interval such that, for some $t_0>0$, the spectrum of $H+tP$ is purely point in $J$ for all $t\in[0,t_0)$. Then $\sigma(H)\cap J$ does not contain an interval.
\end{thm}
\begin{rem}
\label{rem_allt}
Since $P$ is a rank one operator, the conclusion of the theorem also holds for $H+tP$ for all $t$.
\end{rem}
Clearly, the result immediately follows from Proposition \ref{prop_gordon} if we assume that $e_0$ is a cyclic vector for $H$. The essential part of the proof is making sure that this assumption is satisfied on a sufficiently large subspace for one-dimensional Schr\"odinger operators with point spectra (we also do not a priori assume that the spectrum of $H$ is simple, although the spectrum on $J$ will be such since it is purely point).
\subsection{Structure of the paper}In Section 2, we prove Theorem \ref{th_main}. Afterwards, we introduce the complete setting for quasiperiodic operators with $\gamma$-monotone potentials and prove Theorem \ref{th_main_cantor}, which implies Theorem \ref{th_main_small} from the Introduction. In Section 3, we discuss Schr\"odinger operators with infinite coupling, which are used in the proof of Theorem \ref{th_gap_filling_small}. In Section 4, we state the main results on gap filling (Theorems \ref{th_gap_filling} and \ref{th_gap_filling_2}), which imply Theorem \ref{th_gap_filling_small}. 

The proofs of the gap filling results are based on rational approximations, for which it can be established by a topological argument (Proposition \ref{prop_winding_number}). The main difficulty is continuity, since the usual arguments such as \cite{avron} rely on compactness. For the operator family \eqref{eq_h_def}, two obstructions to compactness are related to the presence of discontinuities and the fact that the potential may be unbounded. In Subsections 4.3 and 4.4 we resolve both of these issues by constructing an extended operator family, using the results of Section 3 to deal with unbounded potentials and auxiliary Cantor-type set in the parameter space in order to deal with discontinuities.

\subsection{Acknowledgements}I. K. was supported by the NSF grants DMS--1846114, DMS--2052519, and the 2022 Sloan Research Fellowship. L. P. was supported by the EPSRC grant EP/V051636/1 and the Leverhulme Trust grant RPG-2023-325. R. S. was supported by the NSF grant DMS-2306327.

The authors would like to thank S. Jitomirskaya for valuable discussions and the anonymous referees for useful suggestions and remarks.

\section{Absence of intervals in spectra}
We will start from proving Theorem \ref{th_main}. Afterwards, we will establish its consequences for Schr\"odinger operators with $\gamma$-monotone potentials, the most notable being Theorem \ref{th_main_cantor} and Corollary \ref{cor_cantor}.

\subsection{Proof of Theorem \ref{th_main}} In the following lemma,
$\one_J(H)$ denotes the spectral projector of a self-adjoint operator $H$ associated to an interval $J\subset \R$, and $e_0$ is the standard basis vector of $\ell^2(\Z)$ associated to the origin. Let
$$
\Z_+:=\Z\cap[0,+\infty),\quad \Z_-:=\Z\cap(-\infty,0].
$$
Suppose that $\varphi$ is an eigenfunction of $H$ with $\varphi(0)=0$. Then, the functions
$$
\varphi_{\pm}:=\varphi|_{\ell^2(\Z_{\pm})},
$$
are eigenfunctions of the respective half-line restrictions of $H$. Since an eigenfunction of $H$ cannot vanish on a half-line, one has
\bee
\label{eq_max}
\max\{\|\varphi_+\|,\|\varphi_-\|\}\le \sqrt{1-\delta^2}\|\varphi\|,\quad \delta=\delta(\varphi)=\frac{1}{\|\varphi\|}\min\{\|\varphi_+\|,\|\varphi_-\|\}>0.
\ene
In our notation, $\|\cdot\|$ will always denote the $\ell^2$-norm.
\begin{lem}
\label{lemma_main}
Let $H$ be a discrete one-dimensional Schr\"odinger operator \eqref{eq_h_def}. Suppose that $J\subset \R$ is an open interval such that $\sigma(H)\cap J$ is pure point and $\one_J(H) e_0=0$ $($in other words, every eigenfunction with eigenvalue in $J$ vanishes at the origin$)$. Let $\varphi$, $\psi$ be two distinct normalized eigenfunctions with eigenvalues in $J$:
$$
H\varphi=\lambda\varphi, \,\,H\psi=\mu\psi, \,\,\lambda,\mu,\in J,\,\,\lambda\neq \mu,\,\,\|\varphi\|=\|\psi\|=1,\,\,\varphi(0)=\psi(0)=0.
$$
Then
\bee
\label{eq_m_max}
m:=\min\l\{\max\{\|\varphi_-\|,\|\psi_+\|\},\max\{\|\varphi_+\|,\|\psi_-\|\}\r\}\le\max\{\|\varphi_+\|,\|\varphi_-\|\}\le \sqrt{1-\delta^2},
\ene
where $\delta=\delta(\varphi)>0$ does not depend on $\mu$. Moreover,
\bee
\label{eq_lemma_main_conclusion}
\dist(\lambda,\partial J)\le \frac{|\lambda-\mu|}{\sqrt{1-m^2}}.
\ene
\end{lem}
\begin{proof}
Note that the inequalities in \eqref{eq_m_max} follow from \eqref{eq_max}. Without loss of generality, assume that $m=\max\{\|\varphi_-\|,\|\psi_+\|\}$. Let
$$
f:=a\varphi_-+b\psi_+,
$$
where $a,b$ are chosen in such a way that $(\Delta f)(0)=0$. Note that such choice is always possible with $a,b\neq 0$. Though not important, one can also choose $a,b\in \R$, since all eigenfunctions can be chosen to be real-valued. Denote
$$
Q:=\one_J(H).
$$
Suppose, $\theta\in \mathrm{Ran}(Q)$ is another eigenfunction of $H$ (with an eigenvalue in $J$). Then all six functions $\varphi_-,\varphi_+, \psi_-,\psi_+,\theta_-,\theta_+$ are mutually orthogonal, either due to disjoint supports, or because of being distinct eigenfunctions of the half-line operator.  We have
$$
Qf=a\<\varphi_-,\varphi\>\varphi+b\<\psi_+,\psi\>\psi=a\|\varphi_-\|^2\varphi+b\|\psi_+\|^2\psi,
$$
$$
\frac{\|Q f\|^2}{\|f\|^2}=\frac{a^2\|\varphi_-\|^4+b^2\|\psi_+\|^4}{a^2\|\varphi_-\|^2+b^2\|\psi_+\|^2}\le m^2<1.
$$
As a consequence,
$$
\|(1-Q)f\|^2\ge (1-m^2)\|f\|^2.
$$
Due to the choice of $f$, we also have
$$
(H-\lambda)f=b(\mu-\lambda)\psi_+,\quad\text{therefore}\quad  \|(H-\lambda)f\|^2\le (\mu-\lambda)^2\|f\|^2.
$$
On the other hand, from the spectral theorem one always has
$$
\|(H-\lambda)f\|^2\ge \dist(\lambda,\partial J)^2\|(1-Q)f\|^2.
$$
Combining the above, we finally obtain
$$
\dist(\lambda,\partial J)^2(1-m^2)\le (\mu-\lambda)^2,
$$
which completes the proof.
\end{proof}
\begin{cor}
\label{cor_interval}
Under the assumptions of Lemma $\ref{lemma_main}$, an eigenvalue of $H$ contained in $J$ cannot be a limit point of other eigenvalues of $H$. As a consequence, $\sigma(H)\cap J$ cannot contain an interval.
\end{cor}
\begin{proof}
Let $\lambda\in J$ be an eigenvalue of $H$ and $\mu_j\to \lambda$. In view of \eqref{eq_m_max}, one has \eqref{eq_lemma_main_conclusion} with $m$ independent of $\mu_j$, which leads to a contradiction.
\end{proof}
{\noindent\it Proof of Theorem $\ref{th_main}$.}
Let $\mathcal C\subset \ell^2(\Z)$ be the cyclic subspace for $H$ corresponding to the vector $e_0$. Recall that in the case of unbounded operators it is defined as
$$
\mathcal C=\overline{\mathrm{span}\{(H-z)^{-1}e_0\colon z\in\C\setminus\R\}}.
$$
Both $\mathcal C$ and $\mathcal C^{\perp}$ are invariant subspaces for $(H-z)^{-1}$ for all $z\in \C\setminus \R$ and for all spectral projections of $H$. In particular, the decomposition $\ell^2(\Z)=\mathcal C\oplus \mathcal C^{\perp}$ defines a decomposition of $H$ into a direct sum of two unbounded self-adjoint operators $H_c\oplus H_0$, with $\sigma(H)=\sigma(H_c)\cup \sigma(H_0)$. Note that this can be applied for any $t$, and both $\mathcal C$ and $H_0$ do not depend on $t$:
$$
H+tP=(H_c+tP)\oplus H_0.
$$
As a consequence, the spectra of both operators in $J$ are purely point for $t\in [0,t_0)$. From Proposition \ref{prop_gordon}, it follows that $\sigma(H_c)$ cannot contain an interval. Let $I$ be a connected component of $J\setminus\sigma(H_c)$. Then, all eigenvectors of $H$ with eigenvalues in $I$ must come from eigenvectors of $H_0$, which are orthogonal to $e_0$ and therefore vanish at the origin. From Corollary \ref{cor_interval}, $\sigma(H)$ cannot contain an interval in $I$, which completes the proof.\,\,\qed

While not required for our main results, we would like to note that the assumption of $\sigma(H)\cap J$ being pure point is, in fact, not necessary in Lemma \ref{lemma_main}. An appropriate modification is described in the following lemma.
\begin{lem}
\label{lemma_main_aux}
Let $H$ be a discrete one-dimensional Schr\"odinger operator \eqref{eq_h_def}, $\lambda<\lambda_1<\lambda_2$. Let $\varphi$ be an eigenfunction with eigenvalue $\lambda$:
$$
H\varphi=\lambda\varphi,\quad \|\varphi\|=1,\quad \varphi(0)=0.
$$
Let $Q:=\one_{[\lambda_1,\lambda_2]}(H)$ be the spectral projection of $H$ associated to the interval $[\lambda_1,\lambda_2]$. Suppose that $Q e_0=0$, and that 
$$
\theta\in \mathrm{Ran}\,Q,\quad  \|\theta\|=1.
$$
Then $\langle \theta_+,\varphi_+\rangle=0$.
\end{lem}
\begin{proof}
First, we will establish the following claim: under the assumptions of the lemma, there exists $\theta'\in \mathrm{Ran}\,Q$, $\|\theta'\|=1$, satisfying
\bee
\label{eq_theta_claim}
|\langle \varphi_+,\theta_+\rangle|\le \frac{\lambda_2-\lambda_1}{\lambda_2-\lambda}|\langle \varphi_+,\theta'_+\rangle|.
\ene
It is easy to see that the conclusion of the lemma follows from iterating the claim. In order to establish the claim, note that
$$
\lambda\langle {\varphi_+,\theta_+}\rangle =\<H\varphi_+,\theta_+\>=\<\varphi_+,H\theta_+\>=\<\varphi_+,(H\theta)_+\>.
$$
On the other hand,
$$
(H\theta)_+=(\lambda_2\theta)_++\tilde\theta_+,\quad\text{where}\quad \tilde\theta=Q\tilde\theta,\quad \|\tilde\theta\|\le (\lambda_2-\lambda_1)\|\theta\|=(\lambda_2-\lambda_1).
$$
Let $\theta':=\frac{\tilde\theta_+}{\|\tilde\theta_+\|}$. Then
$$
|(\lambda-\lambda_2)\<\varphi_+,\theta_+\>|=|\<\varphi_+,\tilde\theta_+\>|\le (\lambda_2-\lambda_1)|\<\varphi_+,\theta'_+\>|,
$$
which implies \eqref{eq_theta_claim} and completes the proof.\,\,\qedhere
\end{proof}
\begin{cor}
Let $H$ be a discrete one-dimensional Schr\"odinger operator \eqref{eq_h_def}. Suppose that $J\subset \R$ is an open interval such that $\one_J(H) e_0=0$. Let $\varphi$ be an eigenfunction of $H$ with the eigenvalue $\lambda\in J$, and let $\theta\in \mathrm{Ran}\,\one_{J\setminus\{\lambda\}}(H)$. Then $\langle \theta_+,\varphi_+\rangle=0$.
\end{cor}
\begin{proof}
Follows from applying Lemma \ref{lemma_main_aux} to intervals of the form $[\lambda+1/n,+\infty)\cap J$ and $(-\infty,\lambda-1/n]\cap J$.
\end{proof}
\begin{cor}
Let $H$ be a discrete one-dimensional Schr\"odinger operator \eqref{eq_h_def}. Suppose that $J\subset \R$ is an open interval such that $\one_J(H) e_0=0$. Let $\varphi$, $\psi$ be two distinct normalized eigenfunctions with eigenvalues in $J$:
$$
H\varphi=\lambda\varphi, \,\,H\psi=\mu\psi, \,\,\lambda,\mu,\in J,\,\,\lambda\neq \mu,\,\,\|\varphi\|=\|\psi\|=1,\,\,\varphi(0)=\psi(0)=0.
$$
Define $m$ as in Lemma $\ref{lemma_main}$. Then
\bee
\label{eq_lemma_main_conclusion_2}
\dist(\lambda,\partial J)\le \frac{|\lambda-\mu|}{\sqrt{1-m^2}}.
\ene
As a consequence, an eigenvalue of $H$ contained in $J$ cannot be a limit point of other eigenvalues of $H$.
\end{cor}

\subsection{Quasiperiodic operators with $\gamma$-monotone potentials}In this section, we will apply Theorem \ref{th_main} to quasiperiodic operators with $\gamma$-monotone potentials. In order to state the most general known results on localization for such operators, define
$$
\beta(\alpha):=\limsup_{k\to +\infty}\frac{\log q_{k+1}}{q_k},
$$
where $\frac{p_k}{q_k}$ is the sequence of continued fraction approximants to $\alpha$. 

In the setting of one-dimensional Schr\"odinger operators, many results are stated in terms of the Lyapunov exponent $L(E)$, which is a non-negative real-valued function of energy associated to a one-dimensional ergdodic operator family. We will not be using the specific definition of $L(E)$ and would like to refer the reader to \cite{JK1,K,JK2} in the setting of monotone potentials and to \cite[Chapter 9,10]{CFKS} for a more general textbook treatment; the relevant properties are also summarized in Remark \ref{rem_localizations} below. In addition to $\gamma$-monotonicity, the condition
\bee
\label{eq_log_integrable}
\int_0^1\log(1+|f(x)|)\,dx<+\infty,
\ene
will be required in order for the Lyapunov exponent $L(E)$ of the operator family $H(x)$ to exist.

The following results, in its final form, is established in \cite{JK2}. 
\begin{prop}
\label{prop_monotone_localization}
Assume that $f$ is $\gamma$-monotone and $\eqref{eq_log_integrable}$ holds. Suppose that $L(E)>\beta(\alpha)$. Then, every polynomially bounded solution of the eigenvalue equation
$$
(H(x)\psi)(n)=E\psi(n),\quad n\in \Z,
$$
decays exponentially. As a consequence, the set
$$
\{E\in \R\colon L(E)>\beta(\alpha)\}
$$
can only support purely point spectrum of the operator family $\eqref{eq_h_def_qp}$.
\end{prop}
\begin{rem}
\label{rem_localizations}
Without much detail, we will discuss several situations in which the above result is applicable. The function $f$ is always assumed to be $\gamma$-monotone.
\begin{enumerate}
	\item $L(\cdot)$ is continuous in $E$ and one always has $L(E)\ge \log(\gamma/2)$, see \cite{JK1}. As a consequence, the conclusion of Proposition \ref{prop_monotone_localization} guarantees complete localization for large $\gamma$.
	\item If $f$ is unbounded, then $L(E)>0$ for Lebesgue almost every $E$ \cite{SS}. The same conclusion holds if $f$ is bounded and only has finitely many discontinuities on each trajectory of the irrational rotation \cite{DK}.
	\item The integrated density of states is Lipschitz continuous. As a consequence, if the conclusion of the previous part holds, then for almost every $x$ the zero set of $L(E)$ does not contribute to spectral measure of $H(x)$ (see \cite{JK1,K,JK2}). If $\beta(\alpha)=0$ (which is weaker than the Diophantine condition), then the previous claim implies localization for almost every $x$ for a class of $\gamma$-monotone potentials which includes sawtooth-type and Maryland-type.
	\item Instead of $[0,1)$, one can also consider $\gamma$-monotone functions on $(0,1]$ (for example, by considering $-f(-x)$ instead of $f(x)$). As a consequence, each function from the family $f_t$ defined in \eqref{eq_ft_def} is $\gamma$-monotone in this sense.
	\item In the case $f(0)=-\infty$, the operator $H(n\alpha)$ will have {\it infinite coupling}. In other words, the infinite value $V(n)=-\infty$ of the potential will enforce the Dirichlet condition $\psi(n)=0$, and the operator will split into two half-line operators. The conclusion of Proposition \ref{prop_monotone_localization} still holds in this case, but the operator may not necessarily have simple spectrum. We discuss the case of infinite coupling in detail in Section 3.
\end{enumerate}
\end{rem}
In order to state the second main result, we will need to be somewhat careful about the dependence of $\sigma(H(x))$ on $x$ as a set. Let $f$ be $\gamma$-monotone. We will call $x_0\in [0,1)$ {\it generic} if, for every $y_0\in (x_0+\alpha\Z)\,\,\mathrm{mod}\,\,1$, we have one of the following:
\begin{enumerate}
	\item $y_0\in (0,1)$ and $f$ is continuous at $y_0$.
	\item $y_0=0$ and $f(y_0)=f(y_0+0)=-\infty$, $f(1-0)=+\infty$.
\end{enumerate}
It is easy to see that $f$ is Maryland-type if and only if every point $x_0\in [0,1)$ is generic and that set of non-generic points of a $\gamma$-monotone function $f$ is at most countable. 
\begin{lem}
\label{lemma_sigma_indepence}
Let $f$ be $\gamma$-monotone. Then, the following holds.
\begin{enumerate}
	\item If $x,y\in [0,1)$ and $x$ is generic, then $\sigma(H(x))\subset\sigma(H(y))$.
	\item As a consequence, $\sigma(H(x))=:\Sigma_f$ does not depend on $x$ for generic $x$.
	\item For every $x\in [0,1)$ we have $\sigma(H(x+0))=\sigma(H(x-0))=\Sigma_f$.
	\item For every $x\in [0,1)$, we have $\sigma_{\mathrm{ess}}(H(x))=\Sigma_f$.
\end{enumerate}
\end{lem}
Most of the claims of the lemma are known, especially if one avoids infinite coupling. We postpone the proof of the Lemma \ref{lemma_sigma_indepence} until the end of Section 3, where these issues can be dealt with systematically. Note that the operator $H(x-0)$ (with the potential $V(n)=f(x+n\alpha-0)$) may not belong to the original family $\{H(x)\}_{x\in [0,1)}$. However, a natural way to extend the family is discussed in Section 4.
\begin{rem}
\label{rem_continuity_zd}
The claims of Lemma \ref{lemma_sigma_indepence} also hold for operators defined on $\Z^d$, see Corollary \ref{cor_monotone_compact} and the proof of Lemma \ref{lemma_sigma_indepence} afterwards in Section 3.
\end{rem}
We are now ready to state the first main result of the paper. As mentioned earlier, the ``if'' part has already been established in \cite{K} and is only included for completeness.
\begin{thm}
\label{th_main_cantor}
Assume that $f$ is $\gamma$-monotone and $\eqref{eq_log_integrable}$ holds. For a fixed $x\in [0,1)$, define $H(x)$ as in \eqref{eq_h_def_qp}. Let $J\subset \R$ be an open energy interval such that $L(E)>\beta(\alpha)$ for all $E\in J$. Then, $\sigma(H(x))\cap J$ contains an interval if and only if $f$ is Maryland-type.
\end{thm}
\begin{proof}
Suppose that $f$ is not Maryland-type, and let $x_0$ be a non-generic point of $[0,1)$ with respect to $f$. Due to Lemma \ref{lemma_sigma_indepence}, it is sufficient to show that $\sigma(H(x_0))\cap J$ does not contain an interval. There exists
$$
y_0=(x_0+m\alpha)\,\,\mathrm{mod}\,\,1\in [0,1)
$$ 
such that one of the following holds

\vskip 1mm

{\noindent\it Case 1:} $f(y_0)<f(y_0+0)$. Let $I:=[f(y_0),f(y_0+0)]$ and consider
$$
f_t(x):=\begin{cases}
f(x),&x\in [0,1)\setminus \{y_0\};\\
t,&x=y_0.
\end{cases}
$$
for $t\in I$. Define $H_t(x)$ to be $H(x)$ with $f$ replaced by $f_t$. For every $t\in I$, $f_t$ is a $\gamma$-monotone function, and the corresponding operator family $H_t(\cdot)$ has the same Lyapunov exponent as $H(\cdot)$. Therefore, Proposition \ref{prop_monotone_localization} is applicable, and $\sigma(H_t(x_0))\cap J$ are pure point. By Theorem \ref{th_main_cantor}, $\sigma(H_t(x_0))\cap J$ cannot contain intervals. For $t=f(y_0)$, this implies that $\sigma(H(x_0))$ cannot contain an interval.

\vskip 1mm
{\noindent\it Case 2:} $y_0\in (0,1)$, $f(y_0-0)<f(y_0)$. One can repeat the construction from Case 1 with $I=[f(y_0-0),f(y_0)]$. Alternatively, one can consider $g(x)=-f(-x)$ as will be later done in Case 4.

\vskip 1mm

{\noindent\it Case 3:} $y_0=0$, $f(y_0)>-\infty$. Similarly, one can repeat Case 1 with $I=[-\infty,f(y_0)]$.

\vskip 1mm

{\noindent\it Case 4:} $y_0=0$, $f(y_0)=-\infty$, $f(1-0)<+\infty$. This case can be reduced to the previous case in two steps. First, replace $f(x)$ by $g(x):=-f(1-x)$. Denote the original operator by $H^f$, and the new operator by $H^g$. It is easy to see that $H^f(x)$ is unitarily equivalent to $-H^g(1-x)$, and $g$ is $\gamma$-monotone on $(0,1]$. Define the new function
$$
u(x):=\begin{cases}
g(x),&x\in (0,1);\\
-f(1-0),&x=0.
\end{cases}
$$
It is easy to see that $u(x)$ is now $\gamma$-monotone on $[0,1)$ and is within the class of operators considered in Case 2 (with $J$ replaced by $-J$). As a consequence, $\sigma(H^u(0))\cap (-J)$ cannot contain an interval. Since $H^u(0)$ is a finite rank perturbation of $H^g(1)$ and the latter is unitarily equivalent to $-H^f(0)$, we have that $\sigma(H^f(0))\cap J$ also cannot contain an interval.
\end{proof}
The following corollary implies Theorem \ref{th_main_small}, in view of Remark \ref{rem_localizations}.
\begin{cor}
\label{cor_cantor}
Let $f$ be $\gamma$-monotone, satisfying $\eqref{eq_log_integrable}$ and not of Maryland-type. Suppose that $L(E)>\beta(\alpha)$ for Lebesgue almost every $E\in \R$. Then $\sigma(H(x))$ does not contain any intervals. As a consequence, $\sigma_{\mathrm{ess}}(H(x))$ is a Cantor set, and $\sigma(H(x))$ is a Cantor set for generic $x$.
\end{cor}
\begin{proof}
As mentioned before (\cite{JK1,K,JK2}), $L(E)$ is continuous. Therefore, one can apply Theorem \ref{th_main_cantor} on each open interval of the set $\{E\colon L(E)>\beta(\alpha)\}$. Since each isolated point of $\sigma(H(x))$ is an eigenvalue of finite multiplicity, we have that $\sigma_{\mathrm{ess}}(H(x))$ is a closed set that does not contain any intervals or isolated points, and therefore is a Cantor set.
\end{proof}
\begin{rem}
As mentioned above, for non-generic $x$ the spectrum $\sigma(H(x))$ may contain isolated eigenvalues in the gaps of the Cantor set $\Sigma_f=\sigma_{\mathrm{ess}}(H(x))$. This would typically happen if, for some $x$, one has $f(x-0)\neq f(x)\neq f(x+0)$ (where one assumes $+\infty=-\infty$). Theorem \ref{th_gap_filling} provides a large class of such examples. 
\end{rem}
\section{Operators with infinite potentials}
In this section, we work with operators \eqref{eq_h_def_unbounded} acting on $\ell^2(\Z^d)$. We will prove Lemma \ref{lemma_sigma_indepence}, as well as lay some groundwork for proving Theorem \ref{th_gap_filling_small} in higher dimensions. Note that neither of these results rely on localization and therefore do not involve $L(E)$ or $\beta(\alpha)$ that are only defined in dimension one anyway.

The proof of Lemma \ref{lemma_sigma_indepence} is fairly standard, but requires certain care when considering strong limits as the values of the potential approach infinity. A similar issue appears when one considers gap filling such as in Theorem \ref{th_gap_filling_small}: an operator with infinite coupling belongs to a family of rank one perturbations that are used to fill the gaps. In the proof of the gap filling result, we will need to use rational approximations and obtain several statements about continuity of spectra, based on the same ideas as in \cite{avron} and later developments, some of which were implemented in \cite{K}. In the case of quasiperiodic operators with discontinuous/unbounded potentials, the usual issue is lack of compactness of the set of operators under consideration, either due to multiple ways of approaching a discontinuity in the limit, or due to the value of the potential approaching infinity.

In order to take full advantage of the above arguments, we will extend the class of operators under consideration, making the parameter space compact in the natural topology. This will include allowing the potential to take infinite values. For the convenience of the reader, we discuss the relevant theory, which we do not believe to be completely novel. The results of this section are obtained either for general Schr\"odinger operators on $\ell^2(\Z^d)$ or, as in Corollary \ref{cor_monotone_compact}, for general monotone quasiperiodic operators on $\ell^2(\Z^d)$ with rationally independent frequency vectors.

In order to avoid overusing of the word ``generalized'', the following conventions will be used without further clarification in this and later sections.
\begin{itemize}
	\item The potential of a Schr\"odinger operator can take values in $\rbar=\R\cup\{\infty\}$, which is identified with a circle. The positive orientation of the circle is inherited from the positive direction on $\R$. There will be no distinction between $V(n)=+\infty$ and $V(n)=-\infty$.
	\item The main topology for these operators will be the {\it $\rbar$-strong topology}. Following Proposition \ref{prop_strong_convergence} and Corollary \ref{cor_compact}, it is equivalent to pointwise convergence of potentials in $\rbar$.
	\item The (generalized) spectrum of a Schr\"odinger operator $H$ is denoted by $\bar\sigma(H)$ and is a closed subset of $\rbar$. We will always have $\sigma(H)=\bar\sigma(H)\cap\R$.
\end{itemize}
\subsection{Potentials with infinite values} \label{subsection_infinite}
Let $\rbar=\R\cup \{\infty\}$ be the extended real line and $V\colon \Z^d\to \rbar$. Denote
$$
I(V):=\{\bn\in \Z^d\colon V(\bn)=\infty\}.
$$
The discrete Schr\"odinger operator with potential $V$ is defined by the expression
\bee
\label{eq_h_def_unbounded}
(H\psi)(\bn)=\one_{\Z^d\setminus I(V)}(\bn)(\Delta\psi)(\bn)+V(\bn)\psi(\bn),\quad \bn\in \Z^d.
\ene
on the domain
$$
\dom(H)=\l\{\psi\in \ell^2(\Z)\colon \sum_{\bn\in\Z}|V(\bn)\psi(\bn)|^2<+\infty\r\}.
$$
Here we assume $0\cdot\infty=0$. The closure of $\dom(H)$ in $\ell^2(\Z^d)$ is $\ell^2(\Z^d\setminus I(V))$. It is convenient to consider $\ell^2(I(V))$
as the eigenspace of $H$ corresponding to the eigenvalue $\infty$, which makes the spectral measure of $H$ a measure on $\rbar$ rather than $\R$. We denote
$$
\bar\sigma(H):=\begin{cases}
	\overline{\sigma(H)},&I(V)=\varnothing;\\
	\sigma(H)\cup\{\infty\},&I(V)\neq\varnothing,
\end{cases}
$$
where $\overline{\sigma(H)}$ means the closure in $\rbar$. Equivalently, one can consider the (generalized) Cayley transform
$$
U_H:=i{(H-i\one_{\Z^d\setminus I(V)})(H+i\one_{\Z^d\setminus I(V)})^{-1}}\oplus i\one_{I(V)}
$$
and define the spectral measure of $H$ on $\rbar$ as the pullback of the spectral measure of $U$ under the map $t\mapsto i\frac{t-i}{t+i}$. As a consequence, the inclusion of $\infty$ into the spectral theory of $H$ can be done by the means of the usual spectral theory of $U_H$.

Denote by $\mathcal H$ the set of all operators $H$ of the form \eqref{eq_h_def_unbounded}, and let
$$
\mathcal U_{H}:=\{U_H\colon H\in \mathcal H\}
$$
be the set of their Cayley transforms. Our goal is to relate pointwise convergence of the potentials with convergence of the spectra of the corresponding Schr\"odinger operators. Since the spectra can be completely described in terms of $U_H$, this can be done by the means of strong operator topology on the set $\mathcal U_{H}$. The key property is the following proposition.
\begin{prop}
\label{prop_strong_convergence}
Let $H_j=\Delta+V_j$, $H=\Delta+V$. Suppose that $V_j(\bn)$ converges to $V(\bn)$ in $\rbar$ for every $\bn\in \Z$. Then $U_{H_j}$ converges to $U_H$ strongly.
\end{prop}
\begin{proof}
Instead of Cayley transforms, one can equivalently work with (generalized) resolvents. If $H\in \mathcal H$, we will use the notation $(H+i\one)^{-1}$ to denote the resolvent of $H$, which is originally a bounded operator on $\ell^2(\Z^d\setminus I(V))$, extended by zero into $\ell^2(I(V))$:
$$
(H+i\one)^{-1}:=(H+i\one_{\Z^d\setminus I(V)})^{-1}\oplus \bze_{I(V)}.
$$
 With this convention, we have
$$
U_H=i\one+2(H+i\one)^{-1}.
$$
As a consequence, the strong convergence of $U_{H_j}$ is equivalent to the strong convergence of the (generalized) resolvents defined above. It is easy to check that the projected resolvent identity
\bee
\label{eq_resolvent_identity}
(H_j+i\one)^{-1}\one_{\Z^d\setminus I(V)}-\one_{\Z^d\setminus I(V_j)}(H+i\one)^{-1}=(H_j+i\one)^{-1}(V-V_j)(H+i\one)^{-1}
\ene
holds with the convention $0\cdot \infty=0$, which eliminates all entries of $V(\bn)-V_j(\bn)$ with $\bn\in I(V_j)\cup I(V)$.

Suppose that $V_j\to V$ pointwise on $\Z^d$. Our goal is to show strong convergence of $U_{H_j}$ or, equivalently, $(H_j+i\one)^{-1}$. Since the resolvents are uniformly bounded, it is sufficient to check the convergence on a dense set. Suppose first that $e_\bn$ is a standard basis vector with $\bn\in I(V)\setminus I(V_j)$. Since $V_j(\bn)\to \infty$, we can assume that, eventually, $V_j(\bn)\neq 0$, and write
$$
(H_j+i\one)^{-1}e_{\bn}=\frac{1}{V_j(\bn)}(H_j+i\one)^{-1}\l((H_j+i\one)e_\bn-ie_{\bn}-\Delta e_{\bn}\r)\to 0\,\,\text{as}\,\,j\to +\infty.
$$
For $\bn\in I(V_j)$, one already has $(H_j+i\one)^{-1}e_{\bn}=0$. By linearity, strong convergence is now established on a dense subset of $\ell^2(I(V))$.

For $\bn\in \Z^d\setminus I(V)$, let $\psi_{\bn}:=(H+iI)e_{\bn}$. Since linear combinations of $e_{\bn}$ are dense in $\dom H$ both in $\ell^2$ and $H$-norms, we have that the linear combinations of $\psi_{\bn}$ are dense in $\ell^2(\Z^d\setminus I(V))$. By applying \eqref{eq_resolvent_identity} to $\psi_{\bn}$ and using $(V-V_j)e_{\bn}\to 0$ and that the resolvents are bounded, we arrive to convergence of $(H_j+i\one)^{-1}$ on a dense subset of $\ell^2(\Z^d\setminus I(V))$, which completes the proof.
\end{proof}
The standard diagonal procedure argument immediately implies the converse statement.
\begin{cor}
\label{cor_compact}
The set $\mathcal H$ is compact in the topology of pointwise $\rbar$ convergence of potentials from Proposition $\ref{prop_strong_convergence}$. As a consequence, the set $\mathcal U$ is also compact in the strong operator topology, and the converse of Proposition $\ref{prop_strong_convergence}$ also holds: strong topology on $\mathcal U$ induces pointwise $\rbar$ convergence for potentials on $\mathcal H$.
\end{cor}
We will refer to this convergence of (generalized) Schr\"odinger operators on $\mathcal H$ as the {\it $\rbar$-strong convergence}. Finally, the following is the standard result of semicontinuity of the spectra.
\begin{cor}
\label{cor_semicontinuity}
Suppose that $H_j\to H$ $\rbar$-strongly, and $E\in \bar\sigma(H)$. Then there exists a sequence $E_j\in \bar\sigma(H_j)$ such that $E_j\to E$ in $\rbar$.
\end{cor}
As in the usual setting, the proof can be performed by contradiction, using the resolvents of $U_{H}$ and $U_{H_j}$ near the point $i\frac{E-i}{E+i}$. Corollary \ref{cor_semicontinuity}, as is well known, states that new spectrum cannot appear in strong limit. A converse statement is not necessarily true, but the following weaker claim is useful, see also \cite{avron}.
\begin{prop}
\label{prop_limit_spectrum}
Suppose that $H_j\to H$ $\rbar$-strongly. Assume also that a sequence $\psi_j\colon \Z^d\to [-1,1]$ converges pointwise to $\psi\colon \Z^d\to [-1,1]$ and satisfies a sequence of eigenvalue equations
$$
H_j\psi_j=E_j\psi_j,\quad E_j\in \rbar,\quad  E_j\to E\in \rbar,\quad \psi_j(0)\ge 1/2.
$$
Then $E\in\bar\sigma(H)$.
\end{prop}
\begin{proof}
Suppose first that $E_j\to\infty$. Then $\psi_j(\bn)\to 0$ for all $\bn\in \Z^d\setminus I(V)$. Since $\psi_j(0)\ge 1/2$, we have that $0\in I(V)$, and therefore $\infty\in \bar\sigma(H)$.

It remains to consider the case $E\in \R$. In this case, $\psi_j(\bn)\to 0$ for all $\bn\in I(V)$, and one can pass to the limit in the eigenvalue equation, thus obtaining $H\psi=E\psi$. By the standard Weyl sequence argument, we have that $\psi\in \bar\sigma(H)$.
\end{proof}

We will now discuss some results related to compactness.
\begin{lem}
\label{lemma_large_infinity}
Let $S\subset \Z^d$, and suppose that $|V(\bn)|\ge M$ for all $\bn\in S$. Let
$$
V_{\infty}:=\begin{cases}
\infty,&\bn\in S\\
V(\bn),&\bn\in \Z^d\setminus S,
\end{cases}
$$
and define $H_{\infty}$ using \eqref{eq_h_def_unbounded} with $V$ replaced by $V_{\infty}$. Then
$$
\|U_H-U_{H_{\infty}}\|\le \frac{2+2^d}{M}.
$$
\end{lem}
\begin{proof}
Recall the projected resolvent identity \eqref{eq_resolvent_identity}:
$$R:=(H_{\infty}+i\one)^{-1}\one_{\Z^d\setminus I(V)}-\one_{\Z^d\setminus I(V_{\infty})}(H+i\one)^{-1}=(H_{\infty}+i\one)^{-1}(V-V_{\infty})(H+i\one)^{-1}=0,
$$
since $V=V_{\infty}$ on $\Z^d\setminus (I(V)\cup I(V_{\infty}))=\Z^d\setminus I(V_{\infty})$. Moreover, both resolvents vanish on $\ell^2(I(V))$. Therefore,
$$
\|U_{H_{\infty}}-U_H\|=\|R-\one_{I(V_{\infty})\setminus I(V)}(H+i\one)^{-1}\one_{I(V_{\infty})\setminus I(V)}\|=\|\one_{I(V_{\infty})\setminus I(V)}(H+i\one)^{-1}\one_{I(V_{\infty})\setminus I(V)}\|.
$$
Suppose that $\supp\psi\subset I(V_{\infty})\setminus I(V)$, and $\varphi(\bn):=V^{-1}(\bn)\psi(\bn)$. Then
$$
\|(H+i\one)^{-1}\psi\|=\|(H+i\one)^{-1}\l((H+i\one)\varphi-i\varphi-\Delta\varphi\r)\|\le (2+2^d)\|\varphi\|\le \frac{2+2^d}{M}\|\psi\|,
$$
since $|V(\bn)|\ge M$ on $\supp\psi$ and $\|(H+i\one)^{-1}\|\le 1$.
\end{proof}
\begin{lem}
\label{lemma_finite_rank_unbounded}
Let $V_1,V_2\colon \Z^d\to \rbar$ be such that $V_1(\bn)=V_2(\bn)$ for all but finitely many $\bn\in \Z^d$. Define $H_1$, $H_2$ as in \eqref{eq_h_def_unbounded} with $V_1$, $V_2$, respectively. Then $U_{H_1}-U_{H_2}$ has finite rank.
\end{lem}
\begin{proof}
In the resolvent identity
$$(H_{1}+i\one)^{-1}\one_{\Z^d\setminus I(V_2)}-\one_{\Z^d\setminus I(V_{1})}(H_2+i\one)^{-1}=(H_{1}+i\one)^{-1}(V_2-V_1)(H_2+i\one)^{-1}
$$
the right hand side is of finite rank. Since $I(V_1)$ and $I(V_2)$ can differ by at most finitely many points, the left hand side is also a finite rank perturbation of $U_{H_1}-U_{H_2}$.
\end{proof}
Clearly, if $V_1$ and $V_2$ are two finite potentials on $\Z^d$ such that $|V_1(\bn)-V_2(\bn)|\to 0$ as $|\bn|\to +\infty$, then $H_1-H_2=(\Delta+V_1)-(\Delta+V_2)$ is compact. In the infinite potential case, in addition to allowing the values of the potentials approach one another, one can also allow both of them approach the infinity.

In order to make it precise, let $V_1,V_2\colon \Z^d\to \rbar$. We will say that $V_1\sim V_2$ if for every $\ep>0$ and $M>0$ there exists $N=N(\ep,M)$ such that $|V_1(\bn)-V_2(\bn)|<\ep$ or $|V_1(\bn)|>M$ and $|V_2(\bn)|>M$ for all $\bn$ with $|\bn|>N$.
\begin{lem}
\label{lemma_compact_unbounded}
Let $V_1,V_2\colon \Z^d\to \rbar$. Define $H_1$, $H_2$ as in \eqref{eq_h_def_unbounded} with $V_1$, $V_2$, respectively. Suppose that $V_1\sim V_2$. Then $U_{H_1}-U_{H_2}$ is compact.
\end{lem}
\begin{proof}
Fix $\ep,M>0$ and find $N$ using the fact that $V_1\sim V_2$. Let
$$
V_j'(\bn):=\begin{cases}
0,&|\bn|\le N\\
\infty,&|\bn|>N\,\,\text{and}\,\,|V_j(\bn)|>M\\
V_j(\bn),&|\bn|>N\,\,\text{and}\,|\,V_j(\bn)|\le M,
\end{cases}
$$
and define $H_j'$, $j=1,2$, accordingly. From Lemmas \ref{lemma_large_infinity}, \ref{lemma_finite_rank_unbounded}, and several applications of the triangle inequality, we arrive to
$$
U_{H_1}-U_{H_2}=(U_{H_1'}-U_{H_2'})+A_{\ep,M}+B_{\ep,M},
$$
where $A_{\ep,M}$ has finite rank, and, say $\|B_{\ep,M}\|\le \frac{2^{d+2}}{M}$. On the other hand, since $I(V_1')=I(V_2')$, the resolvent identity implies that $\|U_{H_1'}-U_{H_2'}\|\le \ep$. We thus have obtained a sequence of finite rank approximations converging to $U_{H_1}-U_{H_2}$ in the operator norm, which completes the proof.
\end{proof}
\begin{cor}
\label{cor_monotone_compact}
Let $f_0\colon [0,1)\to [-\infty,+\infty)$ be $\gamma$-monotone, and define
$$
f_1(x):=f_0(x+0),\quad x\in [0,1),
$$
\bee
\label{eq_h_def_zd}
(H_{f_j}(x)\psi)(\bn)=(\Delta\psi)(\bn)+f_j(x+\bn\cdot\alpha)\psi(\bn),\quad x\in [0,1),\quad \bn\in \Z^d
\ene
with the conventions adopted above for infinite values of $f_j$. Then the difference
$$
U_{H_{f_0}(x)}-U_{H_{f_1}(x)}
$$
is compact, and therefore $\sigma_{\mathrm{ess}}(H_{f_0}(x))=\sigma_{\mathrm{ess}}(H_{f_1}(x))$.
\end{cor}
\begin{proof}
From $\gamma$-monotonicity, it is easy to see that the change in the potential between $f_0$ and $f_1$ is within the assumptions of Lemma \ref{lemma_compact_unbounded}.
\end{proof}
{\noindent \it Proof of Lemma $\ref{lemma_sigma_indepence}$.} Note that all of the claims hold for operators on $\Z^d$ considered in Corollary \ref{cor_monotone_compact}. The first claim follows from the fact that, for generic $x$ and any $y$, there is a sequence $\bn_j$ such that $H(y+\bn_j \cdot \alpha)\to H(x)$ $\rbar$-strongly. The second claim follows from the first one. This argument can also be applied to every $x$ if $f$ is either left or right continuous, which implies the third claim. Finally, Corollary \ref{cor_monotone_compact} implies the fourth claim.
\section{Structure of the hull, rational approximations, and gap filling}
\subsection{General properties of gap filling} In the natural topology, $\rbar$ is homeomorphic to a circle. We will assume that positive orientation is inherited from that of $\R$. For $t_1,t_2\in \rbar$, a positively-oriented arc that starts at $t_1$ and ends at $t_2$, will be denoted by $(t_1,t_2)$ and called an interval. One can also consider closed intervals $[t_1,t_2]$. For example, one has 
$$
(1,-2]=\rbar\setminus(-2,1]=(1,+\infty)\cup\{\infty\}\cup(-\infty,-2];
$$
$$
\quad (0,0)=\rbar\setminus\{0\},\quad [\infty,1)=\{\infty\}\cup (-\infty,1).
$$
If $\mathcal A$ is an arc and $f\colon \mathcal A\to \rbar$ is continuous, we will say that $f$ is monotone if it is locally monotone in coordinate charts preserving orientation (essentially, it corresponds to a usual definition of a monotone function on $\R$, with a convention that passing through $\infty$ happens in the correct direction). If both $\mathcal A$ and $\mathcal B$ are arcs not equal to $\rbar$, then $f\colon \mathcal A\to \mathcal B$ (not necessarily continuous) will be called monotone if it preserves ordering on the arcs induced by the orientation. Note that, for a fixed $f$, whether or not it is monotone may depend on the choice of the arcs $\mathcal A$ and $\mathcal B$, even if all these choices are of the same orientation.

Let $A$ be a self-adjoint operator and $P$ be a rank one projection whose range is cyclic for $A$. Suppose, for simplicity, that $\sigma(A)=\sigma_{\mathrm{ess}}(A)$, and let
$$
A_t:=A+tP.
$$
In applications, $A$ will be a Schr\"odinger operator, and $P$ a projection onto one of the standard basis vectors. Therefore, we will be able to also consider $t=\infty$, by defining
$$
A_{\infty}:=\l.(1-P)A(1-P)\r|_{\ran(1-P)},
$$
in accordance with the conventions in Section 3. Let 
$$
G_j=(g_j^-,g_j^+)\subset \rbar
$$ 
be a spectral gap of $A$, including, possibly, one infinite gap considered as a subset of $\rbar$ (in other words, $(g_j^-,g_j^+)$ are the connected components of the complement of the closure of $\sigma(A)$ in $\rbar$). The general theory of rank one perturbations (see, for example, Chapter 11 of \cite{Simon}) implies that $\sigma(A_t)$ with $t\in \rbar\setminus\{0\}$ consists of $\sigma(A)$ and at most one simple eigenvalue in each $(g_j^-,g_j^+)$. This eigenvalue is described by a continuous monotone function $\lambda_j\colon(t_j^-,t_j^+)\to \rbar$, where $(t_j^-,t_j^+)\subset \rbar\setminus \{0\}$ is a (maximal open) interval on which $(g_j^-,g_j^+)$ contains an eigenvalue: that is,
$$
\lambda_j(t_j^{\pm}\mp 0)=g_j^{\pm}.
$$
The functions $\lambda_j$ are strictly monotone in $t$ (in the sense of maps between arcs in $\rbar$). We also have no eigenvalues in $G_j$ for $t\in \rbar\setminus (t_j^-,t_j^+)$. For $t=t_j^{\pm}$, the operator $A_t$ may or may not have an eigenvalue at the corresponding endpoint $g_j^{\pm}$. In the above notation, we also mean that $A_{\infty}$ has a simple eigenvalue $\lambda=\infty$ with the eigenspace $\ran P$, following the conventions of Section 3.

To summarize the above, as $t$ runs over the interval $(0,0)=(0,+\infty)\cup\{\infty\}\cup(-\infty,0)$, each gap will eventually have one eigenvalue appear at one of its endpoints, move continuously and monotonically along the gap, and then disappear at the other endpoint. Eigenvalues in different gaps may appear and disappear at different times, but no eigenvalue can re-appear in the same gap. The starting and ending point of the range for $t$ being $t=0$ corresponds to the fact that we originally assume that $A$ has no isolated eigenvalues.

\subsection{Statements of the results}We will consider Schr\"odinger operators associated to $1$-periodic maps $f\colon \R\to [-\infty,+\infty)$:
\bee
\label{eq_h_def_qp_zd}
(H(x)\psi)(\bn)=(\Delta\psi)(\bn)+f(x+\bn\cdot\alpha)\psi(\bn),\quad x\in [0,1),\quad \bn\in \Z^d,
\ene
where the operator is considered in the sense of \eqref{eq_h_def_unbounded} for infinite values of $f$. Both $[0,1)$ and $[-\infty,+\infty)$ can be identified with circles, making $f|_{[0,1)}$ a circle map. We will consider a class of $f$ such that the corresponding circle maps have one discontinuity (which, without loss of generality, can be placed at the origin).
\begin{defn}
\label{def_simple_discontinuity}
We say that a $1$-periodic function $f\colon \R\to [-\infty,+\infty)$ has a {\it simple discontinuity} if $f|_{[0,1)}$ extends to a continuous map from from $[0,1]$ to $[-\infty,+\infty]$, with $f(0)\neq f(1-0)$ and at least one value among $f(0)$ and $f(1-0)$ being finite.
\end{defn}
Every sawtooth-type potential defined in the introduction has a simple discontinuity. The Maryland-type ones do not, since the corresponding circle maps will be continuous. Let $\mathcal A$ be a closed arc in $\rbar$ (which is associated to $[-\infty,+\infty)$) whose endpoints are $f(0)$ and $f(1-0)$. Note that there are two choices of such arc (more on this later). Our goal will be to consider a family of operators, parametrized by $t\in \mathcal A$, where the value $f(0)$ will be replaced by $t$. More precisely, for $t\in \mathcal A$, define
$$
f_t(x):=\begin{cases}
f(x),&x\in (0,1);\\
t,&x=0.
\end{cases}
$$
Our goal is to state a gap filling result of the kind described in Theorem \ref{th_gap_filling_small}. In the case of a sawtooth-type potential, say, $f(x)=\{x\}$, there is an obvious choice between two arcs (the points of the arc should complement the range of $f$). In general, since $f$ is not assumed to be monotone, we will need to specify it more carefully. For an arc $\mathcal A$ with endpoints $f(0)$ and $f(1-0)$, let
\bee
\label{eq_arc_parametrization}
\varphi_{\mathcal A}\colon[-1,0]\to \mathcal A,\quad \varphi_{\mathcal A}(-1)=f(1-0),\quad\varphi_{\mathcal A}(0)=f(0)
\ene
be a continuous parametrization of $\mathcal A$. Consider a new map $\tilde f\colon [-1,1)\to [-\infty,+\infty)$ defined by
\bee
\label{eq_f_tilde_def}
\tilde{f}(y):=\begin{cases}
f(y),&x\in [0,1);\\
\varphi_{\mathcal A}(y),&y\in [-1,0).
\end{cases}
\ene
Clearly, if one now identifies $[-1,1)$ and $[-\infty,+\infty)$ with circles, $\tilde f$ will become a continuous circle map. We are now ready to state the gap filling result.
\begin{thm}
\label{th_gap_filling}
Suppose that $f$ has a simple discontinuity. Define $H_t(x)$ by $\eqref{eq_h_def_qp_zd}$ with $f$ replaced by $f_t$, $t\in \mathcal A$. Suppose that the associated circle map $\tilde f$, defined above, is not homotopic to a constant. Then
\bee
\label{eq_gap_filling_main}
\bigcup_{t\in \mathcal A}\bar\sigma(H_t(0))=\rbar.
\ene
\end{thm}
Note that, for every $f$ with a simple discontinuity, at least one choice of the arc $\mathcal A$ will produce $\tilde f$ not homotopic to a constant. As described in the previous section, \eqref{eq_gap_filling_main} holds for any family of rank one perturbations with $\mathcal A$ replaced by $\rbar$ (\cite[Theorem 11.8]{Simon}).

The following can also be obtained as a corollary of Proposition \ref{prop_monotone_localization}, assuming that, in addition to the above, the operator is in the complete spectral localization regime.
\begin{cor}
\label{cor_all_eigenvalues}
Suppose that $f$ and $\mathcal A$ satisfy the assumptions of Theorem $\ref{th_gap_filling}$ and Proposition $\ref{prop_monotone_localization}$. Suppose also that $L(E)>\beta(\alpha)$ for all $E\in \R$. Then, for every $E\in \rbar$, there exists $x\in [0,1)$ and $t\in \mathcal A$ such that $E$ is an eigenvalue of $H_t(x)$.
\end{cor}

The proof of Theorem \ref{th_gap_filling} is based on the fact that a version of gap filling holds for rational approximations of the operator. In the process, we use a version of continuity of the spectra under rational approximations, extending the known results of \cite{avron}. Instead of considering separate cases as in \cite{K}, we use a systematic construction of a compactification of the hull that involves operators with infinite potentials. The fact that gap filling holds for rational frequencies uses a ``topological'' argument: if $u\colon \T\to U(n)$ is a continuous family of unitary matrices that is not homotopic to a constant, then the union of their spectra must be equal to $\T$. As a consequence, one also has the following ``topological'' version of the result of \cite{K} for Maryland-type potentials, which also works in every dimension.
\begin{thm}
\label{th_gap_filling_2}
Suppose that $f\colon \R\to [-\infty,+\infty)$ is $1$-periodic, and the circle map associated to $f|_{[0,1)}$ is continuous and not homotopic to a constant.  Define $H(x)$ as in \eqref{eq_h_def_qp_zd}. Then $\sigma(H(x))=\R$.
\end{thm}

\subsection{Structure of the hull}
The proof of Theorem \ref{th_gap_filling} will be based on rational approximations, which requires some kind of continuity of spectra. The usual arguments of \cite{avron} require compactness of the parameter space, which is not present in the case of potentials with simple discontinuity, for two reasons: unboundedness of the potentials and the presence of discontinuity. The constructions of Section 3 allow to consider unbounded potentials, since $\mathcal U_H$ is compact. However, the presence of discontinuity still requires to extend the parameter space: for example, $H(\bn\cdot\alpha-0)$ does not necessarily belong to the original operator family. Finally, the role of the function $\tilde f$ in Theorem \ref{th_gap_filling} suggests that the extended operator family (that is, the one that includes $H_t(0)$ and all it's translations) can also play a role in the proof.

In this section, we will construct a family of operators that will extend the original family $\{H(x)\colon x\in [0,1)\}$ and serve the purposes described above. It will be convenient to construct an auxiliary Cantor set whose gaps are labeled by $\{\bn\cdot\alpha\}$, $\bn\in \Z^d$. Since we will be dealing only with aspects of topology, the exact construction of such a set would not matter (but some maps will be needed to be fixed in a way continuous with respect to $\alpha$). 

Let $h\colon [0,1]\to [0,2]$ be defined as follows:
$$
h(y):=y+3^{-d}\sum_{\bn\in \Z^d\colon \{\bn\cdot\alpha\}\le y}2^{-|n_1|-|n_2|-\ldots-|n_d|},\quad\text{where}\quad \bn=(n_1,\ldots,n_d).
$$
It will be convenient to consider the range $[0,2]$ of $h$ as a subset of the circle that is identified with, say, $[0,3)$. Denote that circle by $\mathcal S$, so that $h\colon [0,1]\to \mathcal S$. Let also $\mathcal C_{\alpha}$ be the closure of the range of $h$ in $\mathcal S$. We will fix the orientation on $\mathcal S$ so that $h\colon [0,1]\to [0,2]$ is monotone (here, $[0,1]$ and $[0,2]$ are considered as arcs of $\mathcal S$).

The set $\mathcal C_{\alpha}$ is a Cantor set if and only if at least one component of $\alpha$ is irrational.
In all cases, the gaps of $\mathcal C_{\alpha}$ (that is, connected components of $\mathcal S\setminus \mathcal C_{\alpha}$) are in one-to-one correspondence with the numbers $\{\bn\cdot\alpha\}$. Let $\mathcal G_{\bn}$ be the gap associated to $\bn\in \Z$. Note that, if the components of $\alpha$ are rationally independent, different values of $\bn$ may correspond to the same value of $\{\bn\cdot\alpha\}$ and therefore to the same gap.
\begin{lem}
\label{lemma_parameter_space}
Suppose that $f$ has a simple discontinuity. Define $H(x)$ as in $\eqref{eq_h_def_qp_zd}$. Then the closure of the set
\bee
\label{eq_h_family}
\{H(x)\colon x\in [0,1)\}
\ene
in the $\rbar$-strong topology is homeomorphic to $\mathcal C_{\alpha}$. The homeomorphism can be defined on the dense subset
\bee
\label{eq_h_dense_family}
\{H(x)\colon x\in [0,1)\setminus(\Z^d\cdot\alpha)\}
\ene
by the map $H(x)\mapsto h(x)$, and is extended to the closure by continuity.
\end{lem}
\begin{proof}
The closure of \eqref{eq_h_family} can be described by
$$
\{H(x\pm 0)\colon x\in [0,1)\},
$$
where one has $H(x)=H(x+0)$ for all $x$ and $H(x)=H(x-0)$ for $x\in [0,1)\setminus\Z^d\cdot\alpha$. For $s,s_j\in \{+0,-0\}$, it is easy to see that $H(x_j s_j)\to H(xs)$ if and only if $x_j\to xs$ with $x_j\neq x$ for $j$ large enough, or $x_js_j=xs$ for $j$ large enough. The latter is equivalent to the topology on $\mathcal C_{\alpha}$.
\end{proof}

The $\Z^d$-action $x\mapsto x+\bn\cdot\alpha$ can be extended to $\mathcal C_{\alpha}$ by continuity. For $\theta\in \mathcal C_{\alpha}$, we will denote this action in the same way: $\theta\mapsto \theta+\bn\cdot\alpha$. Note that, even though $\mathcal C_{\alpha}$ is a subset of a circle, the action is not a circle rotation (and, in general, cannot be continuously conjugated to one). For $\theta\in \mathcal C_{\alpha}$, we will define the corresponding element of the closure of \eqref{eq_h_dense_family} by $H(\theta)$. 

We are deliberately using the same notation for $H(x)$ and $H(\theta)$, meaning the following: the set $[0,1)$ of values of $x$ is naturally identified with the set $\{H(x+0)\colon x\in [0,1)\}$ which is identified with a subset of $\mathcal C_{\alpha}$ of values of $\theta$.
\begin{prop}
Suppose that $\alpha$ has at least one irrational component. There exists $\Sigma\subset \rbar$ such that $\bar\sigma(H(\theta))=\Sigma$ for all $\theta\in \mathcal C_{\alpha}$.
\end{prop}
\begin{proof}
Follows in a standard way from Proposition \ref{prop_strong_convergence}, since the trajectories of the action are dense in $\mathcal C_{\alpha}$ (one can apply the usual proof to the family $\{U_{H(\theta)}\colon \theta\in \mathcal C_{\alpha}\}$ in order to avoid dealing with infinite potentials).
\end{proof}
\subsection{Rank one perturbations and the extended operator family}
As in the previous section, suppose that $f\colon [0,1)\to \rbar$ has a simple discontinuity. Assume that $\mathcal A$ is an arc with endpoints $f(0),f(1-0)$. As in the statement of Theorem \ref{th_gap_filling}, let
$$
f_t(x):=\begin{cases}
f(x),&x\in (0,1);\\
t,&x=0,
\end{cases}
$$
for $t\in \mathcal A$. Let also, as before, $H_t(x)$ be the operator \eqref{eq_h_def_qp_zd} with $f$ replaced by $f_t$.

Recall that 
$$
\varphi_{\mathcal A}\colon[-1,0]\to \mathcal A,\quad \varphi_{\mathcal A}(-1)=f(1-0),\quad\varphi_{\mathcal A}(0)=f(0),
$$
defined in \eqref{eq_arc_parametrization}, is a continuous parametrization of $\mathcal A$. Note that the change of the value $f(0)$ by $t$ only affects the operators $H(x)$ with $x=\bn\cdot\alpha\pm 0$ and interpolates between $H(\bn\cdot\alpha-0)$ and $H(\bn\cdot\alpha+0)=H(\bn\cdot\alpha)$. These two operators correspond to the endpoints of the gap $\mathcal G_{\bn}$ of $\mathcal C_{\alpha}$. We will use the points of $\mathcal G_\bn$ to parametrize this interpolation. More precisely, for each $\bn$, let 
$$
\varphi_{\bn}\colon [-1,0]\to \overline{\mathcal G_{\bn}}
$$
be a monotone parametrization of the gap $\mathcal G_{\bn}$ the appropriately rescaled arc length (where the orientation and arc length on $\mathcal S$ are determined from its identification with $[0,3)$ as done in the definition). For $\theta\in \mathcal G_{\bn}$, define 
$$
H(\theta):=H_{\varphi_{\mathcal A}(\varphi_{\bn}^{-1}(\theta))}(\varphi_{\bn}(1)).
$$
More precisely, it is defined through replacing $f(0)$ by $\varphi_{\mathcal A}(\varphi_{\bn}^{-1}(\theta))$ in the expression for the corresponding matrix element of the operator:
$$
H(\theta)=\l(H(\bn\cdot\alpha-0)-f(1-0)\langle e_{-\bn},\cdot\rangle e_{-\bn}\r)+\varphi_{\mathcal A}(\varphi_{\bn}^{-1}(\theta))\langle e_{-\bn},\cdot\rangle e_{-\bn},\quad \theta\in \overline{\mathcal G_{\bn}}.
$$
The result of this construction is a family of operators $\{H(\theta)\colon\theta\in \mathcal S\}$, with the map $\theta\mapsto H(\theta)$ being a homeomorphism between $\mathcal S$ and the above set of operators in the $\rbar$-strong topology. The action of $\Z^d$ naturally extends to $\mathcal S$ from the conjugation action on $H(\theta)$.
\begin{rem}
The family $\{H(\theta)\colon \theta\in \mathcal S\}$ is an extension of the family $\{H(\theta)\colon \theta\in \mathcal C_{\alpha}\}$ which, in turn, is the closure of the family $\{H(x)\colon x\in [0,1)\}$. As mentioned above, the latter inclusion allows to uniquely associate some value $\theta\in \mathcal C_{\alpha}\subset \mathcal S$ to each value $x\in[0,1)$.
\end{rem}
For $t\in \mathcal A$, denote by 
$$
\mathcal S_{\alpha,t}:=\{\varphi_{\bn}(\varphi_{\mathcal A}^{-1}(t))\colon \bn\in \Z^d\}
$$
the set of parameters corresponding to operators with the same choice of $t$. It is easy to see that $\mathcal S_{\alpha,t}$ are invariant with respect to the translation, and
$$
\cup_{t\in \mathcal A} \mathcal S_{\alpha,t}=\cup_{\bn\in\Z^d}\overline{\mathcal G_{\bn}},\quad \mathcal S\setminus\l(\cup_{t\in \mathrm{int}\, \mathcal J} \mathcal S_{\alpha,t}\r)=\mathcal C_{\alpha}.
$$
In other words, $S_{\alpha,t}$ contains one point in each gap $\mathcal G_{\bn}$. If $\alpha$ has at least one irrational component,  these gaps are dense, and we have
$$
\overline{\mathcal S_{\alpha,t}}=\mathcal S_{\alpha,t}\cup \mathcal C_{\alpha}.
$$
\begin{rem}
\label{remark_different_theta}
It is easy to see that, in the case of a simple discontinuity, all operators $H(\theta)$ are different for different $\theta\in \mathcal S$. If one repeats the steps of the proof for the ``trivial'' case of $f$ such as the associated circle map is continuous (as in Theorem \ref{th_gap_filling_2}), the homeomorphism claim in Lemma \ref{lemma_parameter_space} will fail, since the operators at the endpoints of the gaps will actually be equal. The set \eqref{eq_h_family} will be closed, and there will be no need to add new points.
\end{rem}
\subsection{Continuity of spectra}
We will now discuss continuity in $\alpha$. Whenever the dependence on $\alpha$ is important, we will use the notation $H(\alpha,\theta)$.

\begin{lem}
\label{lemma_strong_continuity}
Suppose that $\theta_j\to \theta\in \mathcal S$, and $\alpha_j\to\alpha$, where the components of $\alpha$ are rationally independent. Then $H(\alpha_j,\theta_j)\to H(\alpha,\theta)$ $\rbar$-strongly. If $\theta_j\in \mathcal S_{\alpha_j,t_j}$, then either $\theta\in \mathcal C_{\alpha}$ or $\theta\in \mathcal S_{\alpha,t}$ and $t_j\to t$. If $\theta\in \mathcal C_{\alpha}$, then one can replace $\theta_j$ by (any) closest to it point of $\mathcal C_{\alpha_j}$, thus making $\theta_j\in \mathcal C_{\alpha_j}$.
\end{lem}
\begin{proof}
Follows from the fact that, for each $\bn\in \Z^d$, the endpoints of each gap $G_\bn$ are continuous in $\alpha$ at all $\alpha$ with rationally independent components.
\end{proof}
\begin{rem}
\label{rem_rational_limits}
In the case when $\alpha$ has rationally dependent components, the conclusion of Lemma \ref{lemma_strong_continuity} may not be true, and the set of such limit points is more complicated. In this case, multiple gaps with different labels are combined into one gap. If $\bn\cdot\alpha\in \Z$, then every gap whose label is a multiple of $\bn$ becomes a part of the gap with label $\bze$. If $\alpha_j$ have rationally independent components, one can choose $\theta_j$ between the gaps with labels $\bze$ and $\bn$. Due to compactness, the operators $H(\theta_j)$ must have a limit point. It is easy to see that it does not correspond to any point of $\mathcal S$. For example, if $d=1$ and $\alpha_j\to \alpha=p/q$, the limit point of $H(\alpha_j,\theta_j)$ may not be a $q$-periodic operator, since the choices of the $\pm$ will not necessarily be $q$-periodic.
\end{rem}
We will now summarize some properties of invariance of spectra.
\begin{enumerate}
	\item For $\alpha$ with at least one irrational component, the spectrum of $H(\theta)$ is constant in $\theta\in \mathcal C_{\alpha}$. We will denote it by $\Sigma(\alpha)$.
\item For all $\alpha$, the spectrum $H(\theta)$ is constant in $\theta\in \mathcal S_{\alpha,t}$ (since the corresponding operators are obtained from one another by translation and therefore are unitarily equivalent). We will use the notation $\Sigma(\alpha,t):=\sigma(H(\theta))$ for $\theta\in \mathcal S_{\alpha,t}$. Since $H(\theta)$ is a rank one perturbation of $H(\theta')$ with some $\theta'\in \mathcal C_{\alpha}$ and the spectrum of the latter has no isolated points, we have that $\Sigma(\alpha,t)$ contains $\Sigma(\alpha)$ and at most one eigenvalue in each gap of $\Sigma (\alpha)$. The notation $\Sigma(\alpha,t)$ is used in order to reflect the fact that, in general, these eigenvalues will not be constant in $t$.
\item For $\alpha$ with rationally dependent components, the spectrum may not be constant in $\theta$. We will denote it by $\Sigma(\alpha,\theta)$.
\end{enumerate}

With all the preparations, the proof of the following continuity result can now be performed in the same language as in \cite{avron}.
\begin{thm}
\label{th_continuity_of_spectra}
Suppose that the components of $\alpha$ are rationally independent, and let $\alpha_j\to \alpha$ be a sequence of vectors with all rational components. Then the following holds:
\begin{enumerate}
	\item $E\in \Sigma(\alpha)$ if and only if there exists $E_j\in \Sigma(\alpha_j,\theta_j)$ with $\theta_j\in \mathcal C_{\alpha_j}$, $E_j\to E$.
	\item $E\in \Sigma(\alpha,t)$ if and only if there exists $E_j\in \Sigma(\alpha_j,t_j)$ with $t_j\to t$ and $E_j \to E$.
	\item As a consequence, the union spectra $\cup_{\theta\in \mathcal C_{\alpha}}\bar\sigma(H(\alpha,\theta))$ and $\cup_{\theta\in \mathcal S}\bar\sigma(H(\alpha,\theta))$ are continuous in $\alpha$ at all rationally independent points, with respect to the Hausdorff distance topology.
\end{enumerate}
\begin{proof}
As in \cite{avron}, the ``only if'' part in (1) and (2) is standard. In both cases, Lemma \ref{lemma_strong_continuity} allows to find a sequence of operators with rational $\alpha_j$ converging to $H(\alpha,\theta)$ $\rbar$-strongly, and the implication follows from Corollary \ref{cor_semicontinuity}. 

For the ``if'' part in (1), note that, since $f$ cannot be identically equal to $\infty$, for large enough $j$ one will be able to find a sequence of bounded solutions of the eigenvalue equation
\bee
\label{eq_eigenvalue_equation}
H(\alpha_j,\theta_j')\psi=E\psi,\quad \|\psi\|_{\ell^{\infty}(\Z^d)}\le 1, \quad \psi(0)\ge 1/2,
\ene
where $\theta_j'\in \mathcal C_{\alpha_j}$ is a translation of $\theta_j$. By the diagonal process, one can assume that $\theta_j'\to\theta\in \mathcal C_{\alpha}$ and $\psi_j\to \psi$ pointwise, with the same properties. From Proposition \ref{prop_limit_spectrum}, this implies $E\in \bar\sigma(H(\theta))$.

A similar argument works for (2), where one can find a converging sequence $\theta_j'\in \mathcal S_{\alpha_j,t_j}$ and a sequence $\psi_j$ satisfying \eqref{eq_eigenvalue_equation}. As in Lemma \ref{lemma_strong_continuity}, the limit of $\theta_j'$ will be either from $\mathcal S_{\alpha,t}$ or from $\mathcal C_{\alpha}$, which in both cases implies $E\in \Sigma(\alpha,t)$.

The proof of (3) goes along the same lines as (1), (2), as well as in \cite{avron}. If $E\in \cup_{\theta\in \mathcal C_{\alpha}}\bar\sigma(H(\alpha,\theta))$, one can find a sequence of rational approximants to $\alpha$ and conclude, by a similar diagonal process, that there exists $\theta\in \mathcal C_{\alpha}$ and a bounded non-trivial solution of the eigenvalue equation $H(\theta)\psi=E\psi$. Then, continuity follows from another application of the same diagonal process. The same holds for $\mathcal C_{\alpha}$ replaced by $\mathcal S$.
\end{proof}
\begin{rem}
\label{rem_rational_fail}
For $\alpha$ with rationally dependent components, the argument in (3) above appears to fail by the same reason as in Remark \ref{rem_rational_limits}.
\end{rem}
\end{thm}
\subsection{The case of rational $\alpha$}For rational $\alpha$ (that is, $\alpha$ with all components being rational), the operator $H(\theta)$ becomes a periodic operator. Suppose that, additionally, $\alpha=(p_1/q_1,\ldots,p_d/q_d)$, with $q_1,\ldots,q_d$ pairwise coprime. Let $q=q_1\ldots q_d$. There exists $\bn_q\in\Z^d$ such that $\bn_q\cdot \alpha=1/q\,\,\mathrm{mod}\,\,\Z$. The set $\mathcal C_{\alpha}$ becomes a union of $q$ arcs, with the set of gap labels being $0,1/q,2/q,\ldots,(q-1)/q$. The operators $H(\theta)$ and $H(\theta+1/q)$ are unitarily equivalent (here, $H(\theta+1/q)$ is defined using the translation by $\bn_q$). In terms of variable $x$, we also have $H(x)$ unitarily equivalent to $H(x+1/q)$, as well as $H(x-0)$ to $H(x-0-1/q)$. It is also convenient to represent $\mathcal C_{\alpha}$ in terms of variable $x$ as a disjoint union
$$
[0,1/q]\sqcup [1/q,2/q]\sqcup\ldots \sqcup[(q-1)/q,1],
$$
where the point $1/q$ on the interval $[0,1/q]$ denotes $H(1/q-0)$, and the one on $[1/q,2/q]$ corresponds to $H(1/q+0)$. 

Let $h(\theta)$ be $H(\theta)$ restricted to $\{0,1,2,\ldots,q_1-1\}\times\ldots\times \{0,1,2,\ldots,q_d-1\}$ with periodic boundary conditions (in other words, the operator with quasimomentum zero). It is well known that $\bar\sigma(h(\theta))\subset \bar\sigma(H(\theta))$. The following is the gap filling result for rational $\alpha$.
\begin{thm}
\label{th_gap_filling_rational}
Suppose that $\alpha=(p_1/q_1,\ldots,p_d/q_d)$, with $q_1,\ldots,q_d$ pairwise coprime. Define $h(\theta)$ as above, and suppose that the map $\tilde f$, defined in \eqref{eq_f_tilde_def}, is of non-zero topological degree. Then
$$
\bigcup_{\theta\in \mathcal S}\bar\sigma(h(\theta))=\rbar.
$$
\end{thm}
Before proceeding to the proof, we will need to use a well known fact about families of unitary operators. Let $u\colon \T\to U(N)$ be a continuous map from the circle to the group of $N\times N$ unitary matrices. The map $t\mapsto \det(u(t))$ is a map from $\T$ to the unit circle in $\mathbb C$.
\begin{prop}
\label{prop_winding_number}
Let $u\colon \T\to U(N)$ be a continuous map from the circle to the group of $N\times N$ unitary matrices, and suppose that the map $t\mapsto \det(u(t))$, considered as a map from $\T$ to $\C\setminus\{0\}$, has a non-zero winding number around the origin. Then
\bee
\label{eq_union_unitary}
\bigcup_{t\in \T}\sigma(u(t))=\T.
\ene
\end{prop}
\begin{proof}
It is easy to see that, if $\eqref{eq_union_unitary}$ does not hold, then there is an open arc of $\T$ such that $\sigma(u(t))$ avoids that arc, and the complement of that arc can be continuously contracted into a point. As a consequence, the map $u$ is homotopic to a constant, and the corresponding determinant map will have a zero winding number.
\end{proof}
{\noindent \bf Proof of Theorem \ref{th_gap_filling_rational}.} Consider the map $\theta\mapsto U_{h(\theta)}$. In view of the above, it is sufficient to show that the corresponding determinant map has a non-zero winding number. Since the said winding number is a topological invariant, one can replace $\Delta$ by $s\Delta$, where $s$ continuously changes from $1$ to $0$ and reduce the original statement to the case of a diagonal operator. In that case, each diagonal entry can be considered as a scalar unitary operator, with the winding number of the determinant equal to that of the map $\tilde f$. As a consequence, each diagonal entry will produce an equal non-zero contribution to the total winding number, which completes the proof.

\vskip 1mm

{\noindent \bf Conclusion of the proofs of Theorems \ref{th_gap_filling} and \ref{th_gap_filling_2}}. Theorem \ref{th_gap_filling} can now be obtained as a consequence of Theorem \ref{th_gap_filling_rational} and part 3) of Theorem \ref{th_continuity_of_spectra}. In view of Remark \ref{remark_different_theta}, the same arguments can also be applies in the simpler case of Theorem \ref{th_gap_filling_2}, with the convention $\mathcal C_{\alpha}=\mathcal S_{\alpha}$ associated with the original set of values of $x\in [0,1)$.\,\,\qed

\vskip 1mm

{\noindent \bf Proof of Corollary \ref{cor_all_eigenvalues}.} 
From the proof of Theorem \ref{th_continuity_of_spectra}, we see that, for every $E\in \R$, there exists $\theta\in \mathcal S$ and a bounded solution to the eigenvalue equation $H(\theta)\psi=E\psi$. If $\theta\in \mathcal C_{\alpha}$, this solution will decay exponentially from Proposition \ref{prop_monotone_localization} applied to the original operator $H$ or the one with $f$ replaced by $f(\cdot-0)$ (see Remark \ref{rem_localizations}). If $\theta$ is in some gap, one can also apply Proposition \ref{prop_monotone_localization} with $f$ replaced by $f_t$.\,\,\qed

\begin{thebibliography}{99}
\bibitem{ten_martini}Avila~A., Jitomirskaya~S., {\it The Ten Martini Problem}, Ann. of Math. (2) 170 (2009), no. 1, 303 -- 342.
\bibitem{ten_martini_dry}Avila~A., You~J., Zhou~Q., {\it Dry Ten Martini Problem in the non-critical case}, preprint (2024), \url{https://arxiv.org/abs/2306.16254}.
\bibitem{avron}Avron~J.,, Simon~B., {\it Almost periodic Schr\"odinger operators II. The integrated density of states}, Duke Math. J. 50 (1983), no. 1, 369 -- 391.
\bibitem{beckus}Band~R., Beckus~S., Loewy~R., {\it The Dry Ten Martini Problem for Sturmian Hamiltonians}, preprint (2024), \url{https://arxiv.org/abs/2402.16703}.
\bibitem{Shi}Cao~H., Shi~Y., Zhang~Z., {\it Localization for Lipschitz monotone quasi-periodic Schr\"odinger operators on $\mathbb Z^d$ via Rellich functions analysis}, preprint (2024), \url{https://arxiv.org/abs/2407.01970}.
\bibitem{CFKS} Cycon~H., Froese~R., Kirsh~W., Simon~B., 
    Schr\"odinger Operators: With Applications to Quantum Mechanics
    and Global Geometry, Springer--Verlag, 1987.
\bibitem{djls}del Rio, R., Jitomirskaya, S. Ya., Last, Y., Simon, B., {\it Operators with singular continuous spectrum, IV. Hausdorff dimensions, rank-one perturbations, and localization}, J. Anal. Math. 69 (1996), 153--200.
\bibitem{DK} Damanik~D., Killip~R., {\it Ergodic potentials with a discontinuous sampling function are non-deterministic}, Math. Res. Lett. 12 (2005), 187 -- 192.
\bibitem{fillman1}Damanik~D., Fillman~J., {\it Schr\"odinger operators with thin spectra}, IAMP News Bulletin, June 2020, \url{https://www.iamp.org/bulletins/old-bulletins/Bulletin-Jul2020-print.pdf}.
\bibitem{DL}Damanik~D., Li~L., {\it Opening Gaps in the Spectrum of Strictly Ergodic Jacobi and CMV Matrices}, preprint (2024), \url{https://arxiv.org/abs/2404.03864}.
\bibitem{drms}del Rio, R., Makarov, N., Simon, B., {\it Operators with singular continuous spectrum. II. Rank one operators}. Comm. Math. Phys. 165 (1994), 59--67.
\bibitem{GJY}Ge~L., Jitomirskaya~S., You~J., {\it Kotani theory, Puig's argument, and stability of The Ten Martini Problem}, preprint (2024), \url{https://arxiv.org/abs/2308.09321}.
\bibitem{G1} Gordon, A. Ya., {\it On exceptional values of the boundary phase for the Schr\"odinger equation on the half-line}, Russ. Math. Surveys 47 (1992), 260--261.
\bibitem{G2}Gordon, A. Ya., {\it Pure point spectrum under 1-parameter perturbations and instability of Anderson localization}. Comm. Math. Phys. 164 (1994), 489--505.
\bibitem{JK1}Jitomirskaya~S., Kachkovskiy~I., {\it All couplings localization for quasiperiodic operators with Lipschitz monotone potentials}, J. Eur. Math. Soc. (JEMS) 21 (2019), no. 3, 777 -- 795.
\bibitem{JK2}Jitomirskaya~S., Kachkovskiy~I., {\it Sharp arithmetic localization for quasiperiodic operators with monotone potentials}, preprint (2024), \url{https://arxiv.org/abs/2407.00703}.
\bibitem{KPS3}Kachkovskiy~I., Parnovski~L., Shterenberg~R., {\it Perturbative diagonalization and spectral gaps of quasiperiodic operators on $\ell^2(\Z^d)$ with monotone potentials}, preprint (2024), \url{https://arxiv.org/abs/2408.05650}.
\bibitem{Xiaowen}Kerdboon~J., Zhu~S., {\it Anderson Localization for Schr\"odinger Operators with Monotone Potentials over Circle Homeomorphisms}, preprint (2023), \url{https://arxiv.org/abs/2305.17599}.
\bibitem{K}Kachkovskiy~I., {\it Localization for quasiperiodic operators with unbounded monotone potentials}, J. Funct. Anal 277 (2019), no. 10, 3467 -- 3490.
\bibitem{Simon}Simon B., Trace Ideals and Their Applications, Second Edition, AMS, 2005.
\bibitem{SS} Simon~B., Spencer~T., {\it Trace class perturbations and the absence of absolutely continuous spectra}, Comm. Math. Phys. 125 (1989), no. 1, 113 -- 125.
\end{thebibliography}
\end{document}